\documentclass{amsart}%
\usepackage{amsmath}
\usepackage{amssymb}
\usepackage{amsfonts}
\usepackage{graphicx}%
\providecommand{\U}[1]{\protect\rule{.1in}{.1in}}
\newtheorem{theorem}{Theorem} [subsection]
\theoremstyle{plain}

\newtheorem{corollary} [theorem] {Corollary}

\newtheorem{lemma} [theorem] {Lemma}

\newtheorem{remark} [theorem]{Remark}

\numberwithin{equation}{subsection}
\begin{document}
\title[Endoscopic factors]{On the structure of endoscopic transfer factors}
\author{D. Shelstad}
\address{Mathematics Department\\
Rutgers University-Newark\\
Newark NJ 07102}
\email{shelstad@rutgers.edu}

\begin{abstract}
There are two versions of endoscopic transfer factors to accommodate the
different versions, classical or renormalized, of the local Langlands
correspondence. An examination of the structure of these complex-valued
factors shows that the versions are related in a simple manner. We gather
various consequences and indicate ways in which these are useful for the
harmonic analysis associated with endoscopic transfer.

\end{abstract}
\maketitle

\section{\textbf{Introduction}}

The purpose of this paper is to examine certain properties of endoscopic
transfer factors. We are guided by the discussion in Sections 4 and 5 of
\cite{KS12} on the expected two versions for the local Langlands
correspondence and some of the implications for transfer.

We begin with transfer factors for twisted endoscopy, specifically the factors
$\Delta^{\prime}$ and $\Delta_{D}$ of \cite{KS12} for the transfer of orbital
integrals. These are assembled in \cite{KS12}, primarily from terms in
\cite{KS99} and \cite{LS87}. Our first goal here is a brief review and then to
establish a simple relation between $\Delta^{\prime}$ and $\Delta_{D}$; see
Corollary 6.1.2 and then Lemmas 6.3.1, 6.3.2 for the particular case of
Whittaker normalization. This allows for quick passage back and forth between
the associated transfers; see (6.4) where we make this explicit. In the
nonarchimedean case the celebrated theorem of Waldspurger proves the existence
of transfer for $\Delta^{\prime}\ $\cite{Wa08}. In the archimedean case,
\cite{Sh12} also uses $\Delta^{\prime}.$

We continue in the archimedean case to the attached spectral transfer,
confining our attention here to tempered representations and standard
endoscopy. The spectral transfer factors of \cite{Sh10} provide coefficients
for the transfer of traces dual to the transfer of orbital integrals. There
are relations that parallel our earlier observations; see Lemmas 8.2.1, 8.3.1.
These are applied to the dual transfer in (8.4) and (8.5).

Our main interest is then the structure on tempered packets provided by the
spectral transfer factors; see (9.1). We also relate the two versions of the
\textit{strong base point property} for tempered packets in the Whittaker
normalization, thereby confirming that the renormalized version is also true;
see Lemma 9.2.1. Our results will be applied to the twisted setting in
\cite{Sh} where we construct twisted spectral factors that provide an
additional structure on packets preserved by the twisting. Compatibility of
the additional structure with the standard one is of significance globally
\cite[Chapter 2]{Ar13}. This provides additional motivation for our approach
in Sections 2 - 4.

Note: A preliminary version of this paper was posted on the author's website
andromeda.rutgers.edu/\symbol{126}shelstad in July, 2012.

\section{\textbf{Setting}}

We consider twisted endoscopy over a local field $F$ of characteristic zero.
Thus $G$ is a connected reductive algebraic group defined over $F,$ $\theta$
is an $F$-automorphism of $G$ and $a$ is a $1$-cocycle of the Weil group
$W_{F}$ in the center $Z_{G^{\vee}}$ of the complex dual group $G^{\vee}.$
There are twisting characters on $G(F)$ attached to $a$ which we label as
$\varpi,\varpi_{D}$ (\cite{KS99}, \cite{KS12} use the notation $\omega
,\omega_{D});$ they are discussed in (4.2). Although we do not use it
explicitly here, there is also an underlying inner twist $(G,\theta,\psi)$ of
a quasi-split pair $(G^{\ast},\theta^{\ast})$, where $\theta^{\ast}$ preserves
an $F$-splitting of $G^{\ast}$; see Appendix B of \cite{KS99} for definitions.
In the transfer theorems we cite, it is assumed that $\theta^{\ast}$ has
finite order.

For a set $\mathfrak{e}=(H,\mathcal{H},s)$ of endoscopic data we follow the
definitions of Chapter 2 of \cite{KS99}, except that we assume harmlessly that
$\mathcal{H}$ is a subgroup of $^{L}G$ and $\xi:\mathcal{H\rightarrow}$
$^{L}G$ is inclusion. We then usually ignore $\xi$ in notation. We also fix a
$z$-pair $(H_{1},\xi_{1})$, again as in \cite{KS99}, and write $\mathfrak{e}%
_{z}$ for the endoscopic data $\mathfrak{e}$ so supplemented. Then $H_{1}$ is
the \textit{endoscopic group} attached to $\mathfrak{e}_{z}$ by which we mean
the group that appears in the associated endoscopic transfer to $G$.

We will call the supplemented data $\mathfrak{e}_{z}$ \textit{bounded }if the
section $c:W_{F}\rightarrow\mathcal{H}$ of p. 18 of \cite{KS99} and the
embedding $\xi_{1}:\mathcal{H\rightarrow}$ $^{L}H_{1}$ \ are such that
$\xi\circ c$, $\xi_{1}\circ c$ are bounded Langlands parameters for $G,H_{1} $
respectively.\textit{\ }It is not difficult to check that given $\mathfrak{e}
$, we may always choose $c$ and $\xi_{1}$ so that this is true. Boundedness of
$\mathfrak{e}_{z}$ will be used in Lemma 5.0.2 and in the notion of a related
pair of bounded Langlands parameters for $H_{1},G$ in (7.2).

There is a map $\mathfrak{e}\rightarrow\mathfrak{e}^{\prime}$ on endoscopic
data which we will use just in the case of standard endoscopy ($\theta=id,$
$a=1)$ where it is particularly simple to describe. Thus suppose that
$\mathfrak{e}=(H,\mathcal{H},s)$ is standard. This returns us to the setting
of \cite{LS87}, although $z$-pairs are used a little differently there. We
will follow \cite{KS99} regarding $z$-pairs; see (3.1) below. From
$\mathfrak{e}$ we obtain another standard set of endoscopic data
$\mathfrak{e}^{\prime}$ by inverting the datum $s$: $\mathfrak{e}^{\prime
}=(H,\mathcal{H},s^{-1}).\ $The same $z$-pair $(H_{1},\xi_{H_{1}})$ serves
$\mathfrak{e}^{\prime}$ and we denote the supplemented data by $\mathfrak{e}%
_{z}^{\prime}. $ The endoscopic group and the point correspondence for
geometric transfer are unchanged. Write $\Delta$ for the relative transfer
factor attached to $\mathfrak{e}_{z}$ in \cite{LS87}; see (3.1). Then we
attach to $\mathfrak{e}_{z}$ another transfer factor $\Delta^{\prime}$ by
taking the factor $\Delta$ for $\mathfrak{e}_{z}^{\prime}.$ This construction
has the same effect as inverting the Tate-Nakayama pairing in the definition
of $\Delta$; see (3.1) again. We may also apply it in other settings; see
(6.2), (6.3).

\subsection{Organizing terms in relative transfer factors}

Terms to be used in the definition of (relative) geometric transfer factors
are described in Section 3\ of \cite{LS87} and Section 4 of \cite{KS99}. One
is a quotient labelled $\Delta_{IV}$ made from discriminants. It may be used
instead in the definition of orbital integrals themselves, specifically in the
normalization of invariant measures on orbits. Lemma 2.12 of \cite{La83} (on
measures for orbital integrals as fiber integrals) provides additional
geometric motivation for doing so. Thus, throughout this paper:

\begin{itemize}
\item \textit{we omit }$\Delta_{IV}$\textit{\ from all formulas for transfer
factors.}
\end{itemize}

Of the other terms, $\Delta_{II}$ is defined concretely and the properties we
will need are read off quickly from lemmas in \cite{LS87} and \cite{KS99}.
There remain $\Delta_{I}$ and $\Delta_{III}$, and in the case of standard
endoscopy $\Delta_{III}$ splits into a product $\Delta_{III_{1}}%
\Delta_{III_{2}}$. These various terms can be gathered into a single term
quite naturally (as remarked in \cite{KS99}) but that is not our concern here.
As defined in \cite{LS87} and \cite{KS99}, each term $\Delta_{I}$,
$\Delta_{III}$, $\Delta_{III_{1}}$, $\Delta_{III_{2}}$ is given by a pairing
in Galois (hyper)cohomology. It will be helpful to review briefly the pairings
involved and their various compatibilities.

\subsection{Pairings in Galois cohomology}

We return to Appendix A of \cite{KS99}. In (A.3.12) a pairing $\left\langle
-,-\right\rangle $ is defined between two groups. The first is the Galois
hypercohomology group $H^{1}(F,T\overset{f}{\longrightarrow}U)$ attached there
to the complex $T\overset{f}{\longrightarrow}U$ of tori over $F,$ concentrated
in degrees $0$ and $1$. The second is the group $H^{1}(W_{F},\widehat{U}%
\overset{\widehat{f}}{\longrightarrow}\widehat{T})$ attached to the complex
$\widehat{U}\overset{\widehat{f}}{\longrightarrow}\widehat{T},$ concentrated
in the same degrees and made from the corresponding $\mathbb{C}$-duals.

In the case that $f$ is trivial, the pairing $\left\langle -,-\right\rangle $
is the product of the classical Langlands pairing\textit{\ }(or of the
\textit{inverse} of the renormalized Langlands pairing $\left\langle
-,-\right\rangle _{D}$) and the Tate-Nakayama pairing as on p. 138 of
\cite{KS99} which we will denote temporarily by $\left\langle -,-\right\rangle
_{tn}$.

In general, the compatibilities for $\left\langle -,-\right\rangle $ are as
follows:
\begin{equation}
\left\langle j(u),\widehat{z}\right\rangle =\left\langle u,\widehat{i}%
(\widehat{z})\right\rangle _{D}^{-1}%
\end{equation}
and%
\begin{equation}
\left\langle z,\widehat{j}(\widehat{t})\right\rangle =\left\langle
i(z),\widehat{t}\right\rangle _{tn}.
\end{equation}
We refer to pp. 137, 138 of \cite{KS99} for unexplained notation. We also
record the compatibility statements for the inverted pairing $\left\langle
-,-\right\rangle _{\ast}$ as%
\begin{equation}
\left\langle j(u),\widehat{z}\right\rangle _{\ast}=\left\langle u,\widehat{i}%
(\widehat{z})\right\rangle ^{-1},
\end{equation}
where on the right we have the classical Langlands pairing for tori, and%
\begin{equation}
\left\langle z,\widehat{j}(\widehat{t})\right\rangle _{\ast}=\left\langle
i(z),\widehat{t}\right\rangle _{tn}^{-1}.
\end{equation}

\section{\textbf{Standard endoscopy}}

\subsection{Factors $\Delta$ and $\Delta^{\prime}$}

As remarked already, we will use the convention of \cite{KS99} for $z$-pairs.
Thus the standard \textit{relative} transfer factor $\Delta=\Delta(\gamma
_{1},\delta;\gamma_{1}^{\dagger},\delta^{\dagger})$ is attached to two very
regular related pairs $(\gamma_{1},\delta)$, $(\gamma_{1}^{\dag},\delta^{\dag
})$ of points in $H_{1}(F)\times G(F).$ In brief, $(\gamma_{1},\delta)$ is a
very regular related pair if $\gamma_{1}$ is strongly $G$-regular, $\delta$ is
strongly regular and the image $\gamma$ of $\gamma_{1}$ in $H(F)$ coincides
with the image of $\delta$ under the inverse $T\rightarrow T_{H}$ of an
admissible isomorphism of maximal tori $T_{H}\rightarrow T,$ where
$T_{H}=Cent(\gamma,H)$ and $T=Cent(\delta,G)$; see \cite{LS87}. When we come
to the twisted case in the next section we will use the generalization of this
formalism to the norm correspondence in Section 3 of \cite{KS99}; there
$T\rightarrow T_{H}$ becomes an \textit{abstract norm} map.

Following our removal of $\Delta_{IV},$ the factor $\Delta$ is a product%
\begin{equation}
\Delta:=\Delta_{I}\Delta_{II}\Delta_{III_{1}}\Delta_{III_{2}}.
\end{equation}
Only the term $\Delta_{III_{1}}=\Delta_{III_{1}}(\gamma_{1},\delta;\gamma
_{1}^{\dag},\delta^{\dag})$ is genuinely relative in the sense that every
other term is defined as a quotient of absolute terms, $\Delta_{I}(\gamma
_{1},\delta)\diagup\Delta_{I}(\gamma_{1}^{\dag},\delta^{\dag})$ and so on.

With the exception of $\Delta_{III_{2}},$ each term depends on the images
$\gamma,$ $\gamma^{\dag}$ of $\gamma_{1},$ $\gamma_{1}^{\dag}$ in
$H(F)=H_{1}(F)\diagup Z_{1}(F)$ rather than on $\gamma_{1},$ $\gamma_{1}%
^{\dag}$ themselves and so, with this exception, we may use the definitions of
Section 3 of \cite{LS87} as written.

Now consider the definition of $\Delta_{III_{2}}(\gamma_{1},\delta)$. If
$H_{1}=H,$ so that $\gamma_{1}=\gamma,$ then we again follow \cite{LS87}.
Namely, fix $\chi$-data $\chi=\{\chi_{\alpha}\}$ for the tori $T_{H}$, $T $ as
in (3.1) of \cite{LS87}. In fact, we will always make in advance all three
choices of that section, namely $\chi$-data, $a$-data and toral data
$T_{H}\rightarrow T,$ whenever we write the terms of transfer factors. Write
$a(\chi)$ for the cocycle $a$ defined on p. 247. It defines an element of
$H^{1}(W_{F},T^{\vee})$ which we also denote by $a(\chi).$ Then
\begin{equation}
\Delta_{III_{2}}(\gamma,\delta):=\left\langle a(\chi),\delta\right\rangle ,
\end{equation}
where $\left\langle -,-\right\rangle $ denotes the classical Langlands pairing
for tori.

Now to drop the assumption $H_{1}=H$, we replace $T$ by the pullback $T_{2}$
of the isomorphism $T\rightarrow$ $T_{H}$ given by toral data and the
projection $T_{1}\rightarrow T_{H}=T_{1}\diagup Z_{1}$; $T_{2}$ is a torus and
is labelled $T_{1}$ on p. 42 of \cite{KS99} where our present $T_{1}$ is
written $T_{H_{1}}.$ Define $\delta_{2}=(\gamma_{1},\delta)\in T_{2}(F)$ and
$a_{2}(\chi)$ to be the class in $H^{1}(W_{F},T_{2}^{\vee})$ of the cocycle
$a_{T}(w)$ on p. 45 of \cite{KS99}. Then
\begin{equation}
\Delta_{III_{2}}(\gamma,\delta):=\left\langle a_{2}(\chi),\delta
_{2}\right\rangle .
\end{equation}

The terms $\Delta_{I}$ and $\Delta_{III_{1}}$ are defined in terms of the
Tate-Nakayama pairing for tori. We see immediately from the definitions in
\cite{LS87} that replacing the endoscopic datum $s$ by $s^{-1}$ inverts these
terms, and that no others are affected. Thus the factor $\Delta^{\prime}$ can
be written as
\begin{equation}
\Delta^{\prime}=(\Delta_{I})^{-1}\Delta_{II}\text{ }(\Delta_{III_{1}}%
)^{-1}\Delta_{III_{2}}.
\end{equation}
In the archimedean case, both $\Delta_{I}$ and $\Delta_{III_{1}}$ are signs,
and so $\Delta^{\prime}$ coincides with $\Delta.$

\subsection{Factor $\Delta_{D}$}

It will be helpful to introduce the factor $\Delta_{D}$ of (5.1) in
\cite{KS12} slightly differently, and we also alter notation. Namely, we now
define the term $\Delta_{III_{2,D}}$ by just replacing the pairing in the
definition of $\Delta_{III_{2}}$ with the renormalized one. In other words,
\begin{equation}
\Delta_{III_{2},D}:=(\Delta_{III_{2}})^{-1}.
\end{equation}
The factor $\Delta_{D}$ is then given as%
\begin{equation}
\Delta_{D}:=\Delta_{I}\text{ }(\Delta_{II})^{-1}\Delta_{III_{1}}%
\Delta_{III_{2,D}}.
\end{equation}
Thus it is immediate that%
\begin{equation}
\Delta_{D}=(\Delta^{\prime})^{-1}.
\end{equation}

The formula (5.3.1) of \cite{KS12} shows that we have not changed the factor
$\Delta_{D}$ itself, at least if $H_{1}=H.$ We review the argument in our
present notation. The terms $\Delta_{II}$ and $\Delta_{III_{2}}$ each depend
on the choice of $\chi$-data but their product does not, and no other term
depends on $\chi$-data; see Lemmas 3.3.D, 3.5.A and the proof of Theorem 3.7.A
of \cite{LS87}, with minor adjustments for the case $H_{1}\neq H$ (or we could
just cite Theorem 4.6.A of \cite{KS99}). Write $\chi^{-1}$ for the $\chi$-data
$\{\chi_{\alpha}^{-1}\}$ and set $\Delta_{II}=\Delta_{II}(\chi)$ when $\chi$
is used in the definition of $\Delta_{II}.$ Clearly,
\begin{equation}
(\Delta_{II}(\chi))^{-1}=\Delta_{II}(\chi^{-1}).
\end{equation}
The term $\Delta_{II}$ also depends on the choice of $a$-data; we assume the
same choice on both sides of (3.2.4).

The use of $\chi$-data in the definition of $\Delta_{III_{2}}$ is more subtle
but we do have the simple relation%
\begin{equation}
\Delta_{III_{2},D}(\gamma,\delta)=\left\langle a_{2}(\chi),\delta
_{2}\right\rangle _{D}=\left\langle a_{2}(\chi),\delta_{2}\right\rangle ^{-1}.
\end{equation}
Because we may replace $\chi$ by $\chi^{-1}$ without affecting $\Delta
_{II}\Delta_{III_{2}}$ we may rewrite $\Delta_{D}$ in a form where only
$\Delta_{III_{2}}$ is changed:
\begin{equation}
\Delta_{D}=\Delta_{I}\Delta_{II}\Delta_{III_{1}}\Delta_{III_{2,D}}(\chi^{-1}).
\end{equation}
Since
\begin{equation}
\Delta_{III_{2,D}}(\chi^{-1})(\gamma,\delta)=\left\langle a_{2}(\chi
^{-1}),\delta_{2}\right\rangle _{D}=\left\langle a_{2}(\chi^{-1}),\delta
_{2}\right\rangle ^{-1},
\end{equation}
this agrees with the definition of $\Delta_{D}$ in (5.1) of \cite{KS12}, which
completes our review.

\section{\textbf{Twisted transfer factors}}

We will be working with the terms $\Delta_{I},$ $\Delta_{II},$ $\Delta_{III}$
as defined in \cite{KS99}, with the following caveat.

\begin{itemize}
\item $\Delta_{I}\ $\textit{is replaced by the term }$\Delta_{I}^{new}%
$\textit{\ of (3.4) in \cite{KS12} without change in notation,}

\item $\Delta_{II}$\textit{\ is exactly as in (4.3) of \cite{KS99},}

\item $\Delta_{III}$\textit{\ is exactly as in (4.4) of \cite{KS99}; in
particular, it is defined in terms of the pairing (A.3.12).}
\end{itemize}

Again we are concerned with relative transfer factors, and as in the standard
case above, only $\Delta_{III}$ is genuinely relative. We will also assume for
convenience that we are in the setting of Chapter 4 of \cite{KS99}, where no
twisting is needed on the endoscopic group. The modifications for the general
case in Sections 5.3 and 5.4 of \cite{KS99} are simple but require a more
detailed look at the norm correspondence.

It is convenient to introduce also the term $\Delta_{III}^{\ast}$ defined via
the inverted pairing $\left\langle -,-\right\rangle _{\ast}$ from (2.2) above.
Of course we have $\Delta_{III}^{\ast}=(\Delta_{III})^{-1}.$ We will also
sometimes write $\Delta_{I}^{\ast}$ for $(\Delta_{I})^{-1}.$

Recall from Section 1 of \cite{KS12} that the product $\Delta_{I}\Delta
_{II}\Delta_{III}$, called $\Delta$ in \cite{KS99}, is \textit{not} a transfer
factor. In other words, the constructions of Section 4 of \cite{KS99} and the
pairing of (A.3.12) do not provide an extension of the factor $\Delta$ of
\cite{LS87} to the twisted case. Following \cite{KS12} we may extend
$\Delta^{\prime}$ by these methods (that this is possible was pointed out by
Waldspurger), and also $\Delta_{D}.$ In (4.1) and (4.2) we give a brief review.

\subsection{Factor $\Delta^{\prime}$}

First, for compatibility with the classical local Langlands correspondence, we
follow \cite{KS12} and introduce%
\begin{equation}
\Delta^{\prime}:=(\Delta_{I}\Delta_{III})^{-1}\Delta_{II}=\Delta_{I}^{\ast
}\Delta_{III}^{\ast}\Delta_{II}.
\end{equation}
In the case of standard endoscopy, the definitions of Section 4.4 of
\cite{KS99} imply that%
\begin{equation}
\Delta_{III}=\Delta_{III_{1}}(\Delta_{III_{2}})^{-1}=\Delta_{III_{1}}%
\Delta_{III_{2,D}}%
\end{equation}
(see also (2.2) above), and so%
\begin{equation}
\Delta^{\prime}=(\Delta_{I}\Delta_{III_{1}})^{-1}\Delta_{III_{2}}\Delta_{II}%
\end{equation}
which coincides with our definition of $\Delta^{\prime}$ using \cite{LS87}.

\textit{Remark:} In \cite{KS12} we corrected an error in \cite{KS99} by
replacing $\Delta$ with $\Delta^{\prime}$. Here we highlight changes for key
arguments in Chapters 4 and 5. First, we replace $\left\langle
-,-\right\rangle $ by $\left\langle -,-\right\rangle _{\ast}$ throughout.
Theorem 4.6.A now states the independence of $\Delta^{\prime}$ from the
choices for $\chi$-data, $a$-data and toral data. It follows because (i)
$\Delta_{III}^{\ast}$ has the properties we wanted for $\Delta_{III}$ in
regard to $\chi$-data and the classical Langlands pairing for tori, (ii) use
of $\left\langle -,-\right\rangle _{\ast}$ does not affect cancellation for
$\Delta_{I}^{\ast},$ $\Delta_{III}^{\ast}$ in regard to toral data, and (iii)
a change in $a$-data causes only sign changes in each of $\Delta_{I},$
$\Delta_{II}$ and so $\Delta_{I}^{\ast},$ $\Delta_{II}$ works just as well as
$\Delta_{I},$ $\Delta_{II}.$ Now consider the results of (5.1) in \cite{KS99}.
If we replace $\Delta$ by $\Delta^{\prime}$ in Lemmas 5.1.A, 5.1.B then the
statements remain correct and the same proofs apply. Lemma 5.1.C and Theorem
5.1.D may be rewritten as statements about the values of relative transfer
factors (we write an analogue for $\Delta_{D}$ explicitly in (4.2.8) below).
Lemma 5.1.C and its generalization at the bottom of p. 53 also concern
characters on certain central subgroups. For the proofs we follow those in
\cite{KS99} and continue to use the classical Langlands pairing for tori.
Next, the formula (1) of Theorem 5.1.D links our factor with the notion of
$\kappa$-orbital integral. The statement now involves $\left\langle
-,-\right\rangle _{\ast},$ and the proof is the same. The formula (2) in
Theorem 5.1.D concerns twisting by a nontrivial quasi-character and is
critical for a well-defined transfer statement for twisted orbital integrals
in that case. The proof is delicate, but we may again follow line by line that
in \cite{KS99} with the caveats we have made. In particular, the parenthetical
remark on p. 56, l. -9 is replaced by \textit{use the pairing} $\left\langle
-,-\right\rangle _{\ast}$, and on the next line we invoke (2.2.3) above in
place of (A.3.13).

\subsection{Factor $\Delta_{D}$}

Next, for compatibility with the renormalized Langlands correspondence, we
have%
\begin{equation}
\Delta_{D}:=\Delta_{I}\Delta_{III}(\Delta_{II})^{-1}.
\end{equation}
It of course extends $\Delta_{D}$ in the standard case, and again
\begin{equation}
\Delta_{D}=(\Delta^{\prime})^{-1}.
\end{equation}
As in the standard case, we can replace $\chi$-data $\chi$ by $\chi^{-1}$
without harm. Thus we get the alternative formula%
\begin{equation}
\Delta_{D}=\Delta_{I}\Delta_{II}\Delta_{III}(\chi^{-1})=\Delta_{I}\Delta
_{II}\Delta_{III}^{\ast}(\chi^{-1})^{-1}.
\end{equation}

For key properties of $\Delta_{D}$ we can appeal to (4.2.2) and the
corresponding properties for $\Delta^{\prime}$ reviewed in the last section
(see Remark). Recall from \cite{KS12} that in place of the two twisting
characters $\varpi,$ $\lambda_{H_{1}}$ we now use their companions $\varpi
_{D},\lambda_{H_{1},D}$ in the various statements. More precisely, the
quasi-character $\varpi$ on $G(F)$ attached on p.17 of \cite{KS99} to the
class $\mathbf{a}$ in $H^{1}(W_{F},Z_{G^{\vee}})$ of our twisting datum $a$ is
replaced by $\varpi_{D}$ given by
\begin{equation}
\varpi_{D}(g):=\left\langle g,\mathbf{a}\right\rangle _{D}=\varpi(g)^{-1}.
\end{equation}
Also the quasi-character $\lambda_{H_{1}}$ on $Z_{1}(F)$ is replaced by
$\lambda_{H_{1},D},$ where
\begin{equation}
\lambda_{H_{1},D}(h_{1}):=\left\langle h_{1},\mathbf{b}\right\rangle
_{D}=\lambda_{H_{1}}(h_{1})^{-1}%
\end{equation}
and $\mathbf{b}$ is the Langlands parameter defined by the homomorphism
\begin{equation}
W_{F}\overset{c}{\longrightarrow}\mathcal{H}\overset{\xi_{_{1}}%
}{\longrightarrow}\text{ }^{L}H_{1}\longrightarrow\text{ }^{L}Z_{1}%
\end{equation}
considered on p. 23. For example, we restate Lemma 5.1.C as%
\begin{equation}
\Delta_{D}(z_{1}\gamma_{1},\delta)=\lambda_{H_{1},D}(z_{1})^{-1}\Delta
_{D}(\gamma_{1},\delta),
\end{equation}
or, in terms of relative factors,
\begin{equation}
\Delta_{D}(z_{1}\gamma_{1},\delta;\gamma_{1},\delta)=\lambda_{H_{1},D}%
(z_{1})^{-1},
\end{equation}
and (2) of Theorem 5.1.D as
\begin{equation}
\Delta_{D}(\gamma_{1},\delta^{\prime})=\varpi_{D}(h)\Delta_{D}(\gamma
_{1},\delta),
\end{equation}
where $\delta^{\prime}=h^{-1}\delta\theta(h)$ and $h\in G(F)$.

\section{\textbf{Two lemmas for relative factors}}

We repeat (4.2.2) as:

\begin{lemma}
For all very regular related pairs $(\gamma_{1},\delta)$, $(\gamma_{1}^{\dag
},\delta^{\dag})$ we have
\[
\Delta_{D}(\gamma_{1},\delta;\gamma_{1}^{\dag},\delta^{\dag})=\Delta^{\prime
}(\gamma_{1},\delta;\gamma_{1}^{\dag},\delta^{\dag})^{-1}.
\]

\end{lemma}

Recall from Section 2 that $\mathfrak{e}_{z}$ is our notation for a set of
endoscopic data supplemented by the choice of a $z$-pair.

\begin{lemma}
For all very regular related pairs $(\gamma_{1},\delta)$, $(\gamma_{1}^{\dag
},\delta^{\dag})$ we have
\[
\left\vert \Delta_{D}(\gamma_{1},\delta;\gamma_{1}^{\dag},\delta^{\dag
})\right\vert =\left\vert \Delta^{\prime}(\gamma_{1},\delta;\gamma_{1}^{\dag
},\delta^{\dag})\right\vert =1
\]
provided $\mathfrak{e}_{z}$ is bounded.
\end{lemma}

\begin{proof}
Consider first the case of standard endoscopy since the argument here points
the way to the general case. Since $\Delta_{I}$ and $\Delta_{III_{1}}$ are
defined by Tate-Nakayama pairings we have $\left\vert \Delta_{I}\right\vert
=\left\vert \Delta_{III_{1}}\right\vert =1$. A check of definitions
\cite{LS87} shows that we may choose the quasi-characters of $\chi$-data to be
unitary. Then $\left\vert \Delta_{II}\right\vert =1.$ It remains to consider
$\Delta_{III_{2}}$. Now because the $\chi$-data are unitary and $\mathfrak{e}%
_{z}$ is assumed bounded, the cocycle $a_{2}(\chi)$ of (3.1) is bounded;
recall this is the cocycle $a_{T}(w)$ constructed on p. 45 of \cite{KS99}.
Then we apply the second part of Theorem 1 of \cite{La97} to conclude that
$\left\vert \Delta_{III_{2}}\right\vert =1,$ and so we are done in the
standard case.

Turning to the general case, we have then only to check that $\left\vert
\Delta_{III}\right\vert =1$ given that $\mathfrak{e}_{z}$ is bounded and the
$\chi$-data are unitary. We could verify an analogue of Langlands' result for
the hypercohomology pairing of (2.2) from \cite{KS99}. This is straightforward
since the arguments of \cite{KS99} are based on those of Langlands: we replace
$\mathbb{C}^{\times}$ by the unit circle $\mathbb{C}_{u}^{\times}$ (notation
of \cite{La97}) at certain steps in the argument. Alternatively, we may return
directly to the standard case by observing that it is enough to prove this for
a sufficiently high power of the pairing. Then it is easy to check in our
application that we may assume that, in the notation of Section 2.2, the first
term in the pairing is in the image of $j. $ We have already noted that the
cocycle $a_{T}(w)$ is bounded. Now the compatibility rules (2.2.1) and (2.2.3)
complete the argument.
\end{proof}

\begin{corollary}
For all very regular related pairs $(\gamma_{1},\delta)$, $(\gamma_{1}^{\dag
},\delta^{\dag})$ we have
\[
\Delta_{D}(\gamma_{1},\delta;\gamma_{1}^{\dag},\delta^{\dag})=\overline
{\Delta^{\prime}(\gamma_{1},\delta;\gamma_{1}^{\dag},\delta^{\dag})}%
\]
provided $\mathfrak{e}_{z}$ is bounded.
\end{corollary}

\section{\textbf{Absolute transfer factors}}

\subsection{Normalization in general}

To reduce the notational burden, we assume the following for the rest of this paper.

\begin{itemize}
\item \textit{The supplemented endoscopic data} $\mathfrak{e}_{z}$ \textit{are
bounded.}
\end{itemize}

\begin{flushleft}
Adjustments for the general case are minor.
\end{flushleft}

Just for this subsection, we write the relative factors $\Delta^{\prime}$ and
$\Delta_{D}$ as $\Delta_{rel}^{\prime}$ and $\Delta_{D,rel}$. By an
\textit{absolute} transfer factor $\Delta^{\prime}$ we mean a complex-valued
function $\Delta^{\prime}$ on all very regular pairs $(\gamma_{1},\delta)$
with the following two properties: $\Delta^{\prime}(\gamma_{1},\delta)\neq0$
if and only if the pair $(\gamma_{1},\delta)$ is related, and
\begin{equation}
\Delta^{\prime}(\gamma_{1},\delta)\diagup\Delta^{\prime}(\gamma_{1}^{\dag
},\delta^{\dag})=\Delta_{rel}^{\prime}(\gamma_{1},\delta;\gamma_{1}^{\dag
},\delta^{\dag})
\end{equation}
whenever each pair $(\gamma_{1},\delta),(\gamma_{1}^{\dag},\delta^{\dag})$ is
related. We define absolute $\Delta_{D}$ analogously.

\begin{lemma}
Suppose that $\Delta^{\prime}$ is an absolute transfer factor for
$\Delta_{rel}^{\prime}.$ Then $(\Delta^{\prime})^{-1},$ $\overline
{\Delta^{\prime}}$ are absolute transfer factors for $\Delta_{D,rel}.$
\end{lemma}

Here by $(\Delta^{\prime})^{-1}$ we mean the factor defined by $(\Delta
^{\prime})^{-1}(\gamma_{1},\delta)=\Delta^{\prime}(\gamma_{1},\delta)^{-1}$
for a very regular \textit{related} pair $(\gamma_{1},\delta)$ and by
$(\Delta^{\prime})^{-1}(\gamma_{1},\delta)=0$ for an unrelated pair.

\begin{proof}
This is immediate from Lemma 5.0.1 and Corollary 5.0.3.
\end{proof}

\begin{corollary}
Let $\Delta^{\prime}$, $\Delta_{D}$ be absolute transfer factors for
$\Delta_{rel}^{\prime}$, $\Delta_{D,rel}$ respectively. Then there exists
$c\in\mathbb{C}^{\times}$ such that
\begin{equation}
\Delta_{D}(\gamma_{1},\delta)=c\text{ }\overline{\Delta^{\prime}(\gamma
_{1},\delta)}%
\end{equation}
for all very regular pairs $(\gamma_{1},\delta).$
\end{corollary}

There are significant cases where $c\neq1.$ Examples are provided by the
Whittaker normalization when we define absolute $\Delta^{\prime}$, $\Delta
_{D}$ using the \textit{same} Whittaker data for each; see Lemma 6.3.1 below.

\subsection{Quasi-split setting}

Assume that $G$ is quasi-split over $F$ and choose an $F$-splitting $spl.$ We
start with the standard case. There are, of course, absolute versions of
$\Delta_{I},$ $\Delta_{II}$ and $\Delta_{III_{2}}.$ This is true for any $G,$
and absolute $\Delta_{I}$ depends additionally on the choice of $spl$ (see 3.2
in \cite{LS87}). Just in our present quasi-split setting, an absolute version
of $\Delta_{III_{1}}$ is defined on p. 248 of \cite{LS87}, and%
\begin{equation}
\Delta_{0}:=\Delta_{I}\Delta_{II}\Delta_{III_{1}}\Delta_{III_{2}}%
\end{equation}
is a transfer factor in the usual sense:
\begin{equation}
\Delta_{0}(\gamma_{1},\delta)\diagup\Delta_{0}(\gamma_{1}^{\dagger}%
,\delta^{\dagger})=\Delta(\gamma_{1},\delta;\gamma_{1}^{\dagger}%
,\delta^{\dagger}),
\end{equation}
for all very regular related pairs $(\gamma_{1},\delta),(\gamma_{1}^{\dagger
},\delta^{\dagger}).$ We may apply the $^{\prime}$-operation to $\Delta_{0}$
to obtain another factor
\begin{equation}
\Delta_{0}^{\prime}:=(\Delta_{I})^{-1}\Delta_{II}(\Delta_{III_{1}})^{-1}%
\Delta_{III_{2}}.
\end{equation}
Similarly, we define%
\begin{equation}
\Delta_{0,D}:=\Delta_{I}(\Delta_{II})^{-1}\Delta_{III_{1}}\Delta_{III_{2,D}}.
\end{equation}
This extends to twisted factors \cite{KS12}. Assume that the twisting
automorphism $\theta$ preserves $spl.$ Then we have the twisted factors%
\begin{equation}
\Delta_{0}^{\prime}:=(\Delta_{I})^{-1}\Delta_{II}(\Delta_{III})^{-1}%
\end{equation}
and%
\begin{equation}
\Delta_{0,D}:=\Delta_{I}(\Delta_{II})^{-1}\Delta_{III},
\end{equation}
where the absolute twisted factor $\Delta_{III}$ is that defined on p. 63 of
\cite{KS99}. It is clear then that
\begin{equation}
(\Delta_{0}^{\prime})^{-1}=\Delta_{0,D},
\end{equation}
provided we define $\Delta_{I}$ in terms of the same $spl\ $in each case.

\subsection{Whittaker normalization}

We continue in the setting of (6.2). The Whittaker normalization was
introduced in Sections 5.2 - 5.4 of \cite{KS99}; see also (5.5) in
\cite{KS12}. It shifts dependence from the $F$-splitting $spl$ preserved by
$\theta$ to a set $\lambda$ of Whittaker data preserved by $\theta$. This is
done by inserting a certain epsilon factor $\varepsilon_{L}(V,\psi)$ \cite[p.
65]{KS99} in the factors $\Delta_{0}^{\prime}$ and $\Delta_{0,D}$. Thus%
\begin{equation}
\Delta_{\lambda}^{\prime}:=\varepsilon_{L}(V,\psi)\Delta_{0}^{\prime
}=\varepsilon_{L}(V,\psi)(\Delta_{I})^{-1}\Delta_{II}(\Delta_{III})^{-1}%
\end{equation}
and%

\begin{equation}
\Delta_{\lambda,D}:=\varepsilon_{L}(V,\psi)\Delta_{0,D}=\varepsilon_{L}%
(V,\psi)\Delta_{I}(\Delta_{II})^{-1}\Delta_{III}.
\end{equation}
A check of the arguments on pp. 65-66 of \cite{KS99} shows that $\Delta
_{\lambda}^{\prime}$ as well as $\Delta_{\lambda,D}$ depends only on the
choice of $\lambda.$

\begin{lemma}%
\begin{equation}
\Delta_{\lambda,D}(\gamma_{1},\delta)=c\text{ }\overline{\Delta_{\lambda
}^{\prime}(\gamma_{1},\delta)}%
\end{equation}
for all very regular pairs $(\gamma_{1},\delta),$ where%
\begin{equation}
c=(det(V))(-1).
\end{equation}

\end{lemma}

Here we regard $det(V)$ as a character on $F^{\times}$; the choice of local
classfield theory isomorphism for this does not matter.

\begin{proof}
Comparing the right sides of (6.3.1) and (6.3.2), we see that this amounts to
a property of $\varepsilon_{L}(V,\psi)$ given on p. 65 of \cite{KS99}, namely
$(\varepsilon_{L}(V,\psi))^{2}=(det(V))(-1).$ The property comes from the fact
that the chosen $V$ is defined over $\mathbb{R}$.
\end{proof}

Next we replace $\lambda$ by $\overline{\lambda}=\lambda^{-1}$ on the left
side of (6.3.3). Here we are using $\lambda$ as notation both for Whittaker
data, namely a $G(F)$-conjugacy class of Whittaker characters, and for a
representative character.

\begin{lemma}%
\begin{equation}
\Delta_{\overline{\lambda},D}(\gamma_{1},\delta)=\overline{\Delta_{\lambda
}^{\prime}(\gamma_{1},\delta)}%
\end{equation}
for all very regular pairs $(\gamma_{1},\delta).$
\end{lemma}

\begin{proof}
Again the needed property is found on p. 65 of \cite{KS99}. With
representative $\lambda$ defined using the $F$-splitting $spl$ and character
$\psi$, replacing $\lambda$ by $\overline{\lambda}$ amounts to replacing
$\psi$ by $\psi_{-1},$ and then $\Delta_{\overline{\lambda},D}(\gamma
_{1},\delta)=c$ $\Delta_{\lambda,D}(\gamma_{1},\delta),$ with $c$ as in
(6.3.4). The lemma now follows.
\end{proof}

\subsection{Consequences for transfer of orbital integrals}

Suppose $\Delta^{\prime}$ and $\Delta_{D}$ are arbitrary normalizations,
\textit{i.e.} are absolute transfer factors for $\Delta_{rel}^{\prime}$,
$\Delta_{D,rel}$ respectively. Assume $\Delta^{\prime}$-transfer exists. Thus
if $f$ is a test function on $G(F)$, \textit{i.e.} $f\in C_{c}^{\infty
}(G(F)),$ then there exists a test function $f_{1}$ on $H_{1}(F)$,
\textit{i.e.} $f_{1}\in C_{c}^{\infty}(H_{1}(F),\lambda_{H_{1}}),$ with
$\Delta^{\prime}$-matching orbital integrals:
\begin{equation}
SO(\gamma_{1},f_{1})=\sum_{\{\delta\}}\text{ }\Delta^{\prime}(\gamma
_{1},\delta)\text{ }O^{\theta,\varpi}(\delta,f)
\end{equation}
for all strongly $G$-regular elements $\gamma_{1}$ in $H_{1}(F).$ Here we
follow the usual conventions for orbital integrals \cite{Wa08}, and the sum is
over $\theta$-conjugacy classes $\{\delta\}$ of strongly $\theta$-regular
elements $\delta$ in $G(F)$.

Recall Corollary 6.1.2 above and notice that the function
\begin{equation}
f_{1,D}:=c\text{ }\overline{(\overline{f})_{1}}%
\end{equation}
lies in $C_{c}^{\infty}(H_{1}(F),\lambda_{H_{1},D}).$ Then, applying complex
conjugation to (6.4.1) with $\overline{f}$ in place of $f$, we see that $f$
and $f_{1,D}$ have $\Delta_{D}$-matching orbital integrals:
\begin{equation}
SO(\gamma_{1},f_{1,D})=\sum_{\{\delta\}}\text{ }\Delta_{D}(\gamma_{1}%
,\delta)\text{ }O^{\theta,\varpi_{D}}(\delta,f)
\end{equation}
for all strongly $G$-regular elements $\gamma_{1}$ in $H_{1}(F).$ Thus
$\Delta_{D}$-transfer exists. Conversely, assume that $\Delta_{D}$-transfer
exists. Then we argue similarly to recover the existence of $\Delta^{\prime}$-transfer.

\textit{Remark:} See \cite{Ka13} for an alternate approach to the case of
Whittaker normalization for quasi-split groups. There is a further comment on
this in (7.4) when we come to the dual spectral transfer for real groups. We
will see that our approach with complex conjugation is sufficient for the
tempered spectrum, both in the Whittaker case and in general, because of
unitarity and the simple relations of Section 5. We expect other applications also.

\textit{Remark:} The arguments of \cite{Sh12} can also be used to prove
\textit{directly} a $\Delta_{D}$-transfer of orbital integrals in the
archimedean case. First, notice that \cite{Sh12} proves transfer for the
factor $\Delta^{\prime}=\Delta_{I}^{\ast}\Delta_{II}\Delta_{III}^{\ast}.$
Since $\Delta_{I}=\Delta_{I}^{\ast}$ in the archimedean setting, proving
transfer for $\Delta_{D}$ then amounts to checking that $\Delta_{III}^{\ast}$
may be replaced by $\Delta_{III}^{\ast}(\chi^{-1})^{-1}$ in the limit formula
of the main lemma (Lemma 9.3). We may certainly invert both sides of that
limit formula without harm, and further it is clear from the definitions in
Section 3 of the paper that we may also invert the two sets of $\chi$-data
simultaneously. Thus we are done.

\section{\textbf{Archimedean case: spectral setting}}

From now on, $F=\mathbb{R}$ and we consider only standard endoscopy. In this
setting, the relative transfer factors $\Delta_{rel}^{\prime}$ and
$\Delta_{rel}$ are the same (recall (3.1) above) and so we may as well
normalize the absolute factors to be the same also. Thus we will use the
notation $\Delta$ from \cite{Sh10} interchangeably with $\Delta^{\prime}$.

\subsection{Statement of dual tempered transfer}

Our interest is in the tempered spectral transfer dual to (6.4.1), the
tempered spectral transfer dual to (6.4.3), and the precise relation between
the two.

We state the two dual transfer formulas as
\begin{equation}
St\text{-}Trace\text{ }\pi_{1}(f_{1})=\sum_{\pi}\text{ }\Delta^{\prime}%
(\pi_{1},\pi)\text{ }Trace\text{ }\pi(f)
\end{equation}
and%
\begin{equation}
St\text{-}Trace\text{ }\pi_{1,D}(f_{1,D})=\sum_{\pi}\text{ }\Delta_{D}%
(\pi_{1,D},\pi)\text{ }Trace\text{ }\pi(f).
\end{equation}
Here $\pi_{1},$ $\pi_{1,D}$ denote tempered irreducible admissible
representations of $H_{1}(\mathbb{R})$ for which the central subgroup
$Z_{1}(\mathbb{R})$ acts by the character $\lambda_{H_{1}},$ $\lambda
_{H_{1},D}$ respectively, and $\pi$ denotes an tempered irreducible admissible
representation of $G(\mathbb{R}).$ Notationally we identify a representation
with its isomorphism class. Also, we may as well limit our attention to
unitary representations on Hilbert spaces (and unitary isomorphism). For each
$\pi_{1}$ there will be only finitely many (isomorphism classes) $\pi$ such
that $\Delta^{\prime}(\pi_{1},\pi)$ is nonzero, and similarly for each
$\pi_{1,D}$ there will be only finitely many $\pi$ such that $\Delta_{D}%
(\pi_{1,D},\pi)$ is nonzero. This is by construction, although it may be
argued directly that this has to be the case. Further, we will see that such
$\pi$ form the packet, classical in the case of $\Delta^{\prime}$,
renormalized in the case of $\Delta_{D},$ predicted by Langlands'
functoriality principle.

The terms $St$-$Trace$ $\pi_{1},$ $St$-$Trace$ $\pi_{1,D}$ on the left sides
denote the stable trace over the packet of $\pi_{1},$ $\pi_{1,D}$
respectively. Each of these is a stable tempered distribution on
$H_{1}(\mathbb{R})$.

The transfer theorems state that for $f,f_{1}$ as in (6.4.1) and $f,f_{1,D}$
as in (6.4.3), the formulas (7.1.1) and (7.1.2) are true when $\Delta^{\prime
}(\pi_{1},\pi),$ $\Delta_{D}(\pi_{1,D},\pi)$ are as we specify. The factor
$\Delta^{\prime}(\pi_{1},\pi)$ is the factor $\Delta(\pi_{1},\pi)$ from
\cite{Sh10} and the associated transfer theorem has been proved there
following the method of \cite{Sh82}. We review briefly and build on this.

As on the geometric side, we first isolate and examine a \textit{very regular}
setting. To this end, we introduce the notion of \textit{very regular
(related) pair} for Langlands parameters.

\subsection{Very regular related pairs}

Suppose that $\varphi:W_{\mathbb{R}}\rightarrow$ $^{L}G=G^{\vee}\rtimes
W_{\mathbb{R}}$ \ is an admissible homomorphism relevant to $G,$ so that the
orbit of $\varphi$ under the conjugation action of $G^{\vee}$ is a Langlands
parameter for $G.$ We use the notation $\boldsymbol{\varphi}$ (boldface) for
the parameter when it is helpful to distinguish it from a representative. We
call $\varphi$ or $\boldsymbol{\varphi}$ \textit{regular} if the subgroup
$Cent(\varphi(\mathbb{C}^{\times}),G^{\vee})$ of $G^{\vee}$ of elements fixing
(pointwise) the image of $\mathbb{C}^{\times}$ is abelian. For the endoscopic
group $H_{1}$ we consider only parameters $\boldsymbol{\varphi}_{1}$ for which
composition with $^{L}H_{1}\rightarrow$ $^{L}Z_{1}$ produces the parameter
$\mathbf{b}$ of (4.2.5). Equivalently, $\boldsymbol{\varphi}_{1}$ has a
(\textit{good}) representative $\varphi_{1} $ with image in $\xi
_{1}(\mathcal{H}),$ as described in Section 2 of \cite{Sh10}. Recall that we
have:%
\begin{equation}%
\begin{array}
[c]{ccc}
&  & ^{L}H_{1}\\
& \overset{\xi_{1}}{\nearrow} & \\
\mathcal{H} &  & \\
& \underset{\xi\text{ }=\text{ }incl}{\searrow} & \\
&  & ^{L}G
\end{array}
\end{equation}
Since $H_{1}\rightarrow H$ is a $z$-extension, good $\varphi_{1}$ is unique up
to conjugation by elements of the subgroup $H^{\vee}$ of $\mathcal{H}.$ The
$G^{\vee}$-orbit of $\varphi^{\ast}=\xi$ $\circ$ $\xi_{1}^{-1}$ $\circ$
$\varphi_{1}$ defines a parameter $\boldsymbol{\varphi}^{\ast}$ for the
quasi-split form $G^{\ast}$ of $G$ (it need not be relevant to $G$). We say
$\boldsymbol{\varphi}_{1}$ is $G^{\vee}$-\textit{regular }if $Cent(\varphi
^{\ast}(\mathbb{C}^{\times}),G^{\vee})$ is abelian.

Assume $\boldsymbol{\varphi}_{1}$ is $G^{\vee}$-regular and
$\boldsymbol{\varphi}$ is a regular parameter for $G$. We call the pair
$(\boldsymbol{\varphi}_{1},\boldsymbol{\varphi})$ a \textit{very regular
pair}. We call $(\boldsymbol{\varphi}_{1},\boldsymbol{\varphi}),$ or
representatives $(\varphi_{1},\varphi),$ \textit{related} if
$\boldsymbol{\varphi}^{\ast}$ is relevant to $G$ and coincides with
$\boldsymbol{\varphi}\mathbf{.}$

\subsection{Parameters in explicit form}

Recall from Section 3 of \cite{La89} a description of essentially bounded
Langlands parameters. Throughout we assume a choice of splitting $spl^{\vee
}=(\mathcal{B},\mathcal{T},\{X_{\alpha^{\vee}}\})$ for $G^{\vee}$ preserved by
the action of the Galois group $\Gamma=\{1,\sigma\}$; the final results are
independent of this choice.

We start with those parameters that factor through no proper parabolic
subgroup of $^{L}G$, namely the parameters for essentially discrete series
representations, provided they exist. Each such parameter has an almost
canonical representative $\varphi=\varphi(\mu,\lambda)$ attached to
$spl^{\vee}$. Here $\mu$ is a uniquely determined element of $X_{\ast
}(\mathcal{T})\otimes\mathbb{C}$ that is integral and regular dominant for
$(\mathcal{B},\mathcal{T)}$ and $\lambda$ is a representative for a uniquely
determined coset in $X_{\ast}(\mathcal{T})\otimes\mathbb{C}$ of the submodule
$\mathcal{K}$ generated by $X_{\ast}(\mathcal{T})$ and the $(-1)$-eigenspace
of the automorphism $\overline{\sigma}$ to be defined shortly. The
\textit{almost canonical} refers to the fact that this representative is
unique only up to the action of $\mathcal{T}$ (see \cite{Sh10} for a slight
extension of the discussion in \cite{La89}).

In particular, for $\varphi=\varphi(\mu,\lambda),$ we have that $\varphi
(\mathbb{C}^{\times})$ is contained in $\mathcal{T}\times\mathbb{C}^{\times}$.
If $w$ has image $\sigma$ under $W_{\mathbb{R}}\rightarrow\Gamma$ then
$\varphi(w)$ normalizes $\mathcal{T}$ and its action on $\mathcal{T}$ is
independent of choice for $w$. We write it as $\overline{\sigma}$. Then
$\overline{\sigma}$ acts on $\mathcal{T}\cap G_{der}^{\vee} $ as inversion
$t\rightarrow t^{-1}.$ By construction,
\begin{equation}
\varphi(z)=z^{\mu}\overline{z}^{\overline{\sigma}\mu}\times z
\end{equation}
for $z\in\mathbb{C}^{\times}.$ Suppose $w^{2}=-1.$ Then $w$ has image $\sigma$
under $W_{\mathbb{R}}\rightarrow\Gamma$ and, by construction again,
\begin{equation}
\varphi(w)\in e^{2\pi i\lambda}G_{der}^{\vee}\times w.
\end{equation}
Here $e^{2\pi i\lambda}$ denotes the element of $\mathcal{T}$ satisfying
\begin{equation}
\lambda^{\vee}(e^{2\pi i\lambda})=e^{2\pi i\text{ }\left\langle \lambda^{\vee
},\text{ }\lambda\right\rangle }%
\end{equation}
for all $\lambda^{\vee}\in X^{\ast}(\mathcal{T}),$ where $\left\langle
-,-\right\rangle $ is the canonical pairing between $X^{\ast}(\mathcal{T})$
and $X_{\ast}(\mathcal{T}).$ Also, $G_{der}^{\vee}$ denotes the derived group
of $G^{\vee}.$

Boundedness of $\varphi$ is a condition on $\mu+\overline{\sigma}\mu$ which we
will impose to allow us to drop the adjective \textit{essentially} from our
discussion; it states that $\left\langle \lambda^{\vee},\mu+\overline{\sigma
}\mu\right\rangle $ is purely imaginary for all $\lambda^{\vee}\in X^{\ast
}(\mathcal{T})\otimes\mathbb{R}$ and forces unitarity of the character data of
the next subsection. Here we use \textit{essentially }as in Section 3 of
\cite{La89}. Thus a discrete series representation has unitary central
character and has matrix coefficients that are square-integrable on
$G(\mathbb{R})\diagup Z_{G}(\mathbb{R})$, where $Z_{G}$ denotes the center of
$G$. An essentially discrete series representation $\pi$ has no condition on
its central character, but $\pi$ then has the property that $\zeta\otimes\pi$
is a discrete series representation, for some quasicharacter $\zeta$ on
$G(\mathbb{R}).$ The \textit{essentially} case is thus no harder to handle,
for discrete series or more generally for tempered representations, but it
does add to the notational burden.

The boundedness condition is of course vacuous if $Z_{G}$ is anisotropic, and
$\overline{\sigma}$ then acts on the entire torus $\mathcal{T}$ as inversion.
In that case we may take $\lambda=0$ and then ignore it. More precisely, in
general we have that
\begin{equation}
\frac{1}{2}(\mu-\overline{\sigma}\mu)-\iota\equiv\lambda+\overline{\sigma
}\lambda\operatorname{mod}X_{\ast}(\mathcal{T}),
\end{equation}
where $\iota$ is one-half the sum of the coroots of $(\mathcal{B}%
,\mathcal{T})$. This is the critical Lemma 3.2 of \cite{La89}. In the case
that the center of $G$ is anisotropic we may replace (7.3.4) by the simpler
rationality condition
\begin{equation}
\mu-\iota\in X_{\ast}(\mathcal{T}).
\end{equation}

Now to drop the assumption that bounded $\boldsymbol{\varphi}$ factors through
no proper parabolic subgroup of $^{L}G,$ we follow the argument of \cite{La89}
to attach a parabolic subgroup of $^{L}G$ standard relative to $(\mathcal{B}%
,\mathcal{T)}$ through which a representative $\varphi$ factors minimally.
Then if $^{L}M$ is the corresponding standard Levi subgroup we may assume that
$\varphi$ has image in $^{L}M$ but factors through no proper parabolic
subgroup of $^{L}M$. Using the subsplitting $spl_{M}^{\vee} $ of $spl^{\vee}$
(defined in the obvious way), we have a description as above, with
$\overline{\sigma}=\overline{\sigma}_{G}$ replaced by $\overline{\sigma}_{M}$
and $\iota=\iota_{G}$ replaced by $\iota_{M}.$

\subsection{Attached packets (discrete series)}

The packet attached to a parameter $\varphi$ in the local Langlands
correspondence depends on the choice of an inner class of inner twists from
$G$ to the fixed quasi-split form $G^{\ast}.$ Thus we choose an inner twist
$\psi:G\rightarrow G^{\ast}$ and are free to adjust it by an inner
automorphism when convenient.

Throughout this subsection we assume that $\varphi$ is bounded and factors
through no proper parabolic subgroup of $^{L}G,$ and write $\Pi$ for the
discrete series packet attached via the classical local Langlands
correspondence of \cite{La89}. We recall how this is done. The existence of
$\overline{\sigma}$ ensures (more precisely, is equivalent to) the existence
of a maximal torus $T$ over $\mathbb{R}$ in $G$ that is anisotropic modulo the
center of $G.$ Then we adjust $\psi$ so that its restriction to $T$ is defined
over $\mathbb{R}$, and use $\psi$ together with toral data to transport
$\mu,\lambda$ to $X^{\ast}(T)\otimes\mathbb{C}$; $\overline{\sigma} $ becomes
$\sigma_{T}.$ Let $\Omega$ be the Weyl group of $T$ in $G.$ We have then a
well-defined collection $\{(\omega\mu-\iota,\lambda);\omega\in\Omega\}$ of
\textit{character data}, \textit{i.e.} one that is independent of the choice
of splittings and toral data, and that depends only on the inner class of
$\psi$. Informally, these character data provide well-defined character
formulas on the regular points of $T(\mathbb{R})$ along the lines of the
familiar Weyl character formula. Harish-Chandra's theorems on the discrete
series allow us to use this information to characterize the distributions
$Trace$ $\pi$ and $St$-$Trace$ $\pi$ for $\pi\in\Pi$ (see Section 3 of
\cite{La89}).

Assuming a choice of toral data, we may now identify $\pi\in\Pi$ by an element
$\omega$ of $\Omega$, or more precisely by a left coset of the subgroup
$\Omega_{\mathbb{R}}$ of elements realized in $G(\mathbb{R})$. We then write
$\pi=\pi(\omega)$; see Section 7b in \cite{Sh10} and below. As in the
definition of the geometric factors, for spectral transfer factors we choose
toral data, $a$-data and $\chi$-data to assemble terms and then observe that
factor itself is independent of these choices.

Before turning to $\Pi_{D},$ we review a critical step in the definition of
the classical packet $\Pi,$ namely forming characters on $T(\mathbb{R})$, to
be denoted $\Lambda_{(\omega\mu-\iota,\lambda)},$ from the character data
$\{(\omega\mu-\iota,\lambda);\omega\in\Omega\}.$ This is done in \cite{La89}
via an explicit form of the classical local Langlands correspondence for real
tori, as follows. The congruence%
\begin{equation}
\frac{1}{2}[(\omega\mu-\iota)-\sigma_{T}(\omega\mu-\iota)]\equiv\lambda
+\sigma_{T}\lambda\operatorname{mod}X^{\ast}(T),
\end{equation}
is a consequence of (7.3.4). Moreover, it is a necessary and sufficient
condition that the formula%
\begin{equation}
\Lambda_{(\omega\mu-\iota,\lambda)}(\exp H\text{ }\exp i\pi\lambda^{\vee
})=e^{\left\langle \omega\mu-\iota,H\right\rangle }\text{ }e^{2\pi
i\left\langle \lambda,\lambda^{\vee}\right\rangle }.
\end{equation}
yields a well-defined character $\Lambda_{(\omega\mu-\iota,\lambda)}$ on
$T(\mathbb{R})$. Here we have identified the Lie algebra $\mathfrak{t}$ of $T$
as $X_{\ast}(T)\otimes\mathbb{C}$ and the Lie algebra $\mathfrak{t}%
_{\mathbb{R}}$ of $T(\mathbb{R})$ as the real subspace of invariants for the
twisted action of $\sigma_{T}$ (meaning $\sigma$ acts on both components). The
element $H$ lies in $\mathfrak{t}_{\mathbb{R}}$ and $\lambda^{\vee}\in
X_{\ast}(T)$ is $\sigma_{T}$-invariant. These elements $\exp i\pi\lambda
^{\vee}$ form a finite subgroup $F$ of $T(\mathbb{R})$ such that
$T(\mathbb{R})=T(\mathbb{R})^{0}$ $F.$

Notice that if $T$ itself is anisotropic over $\mathbb{R}$, \textit{i.e.} the
center $Z_{G}$ of $G$ is anisotropic, then the group $T(\mathbb{R})$ is
connected and the character $\Lambda_{\omega\mu-\iota}$ is just the
restriction of the rational character $\omega\mu-\iota$ to $T(\mathbb{R}). $

In general, consider restriction to $Z_{G}(\mathbb{R})$ of the character
$\Lambda_{(\omega\mu-\iota,\lambda)}.$ Since the group $F$ is central we may
write $z\in Z_{G}(\mathbb{R})$ as $z=\exp H$ $\exp i\pi\lambda^{\vee}, $ with
$H\in\mathfrak{z}_{G,\mathbb{R}}$ and $\lambda^{\vee},$ as above, a
$\sigma_{T}$-invariant cocharacter of $T$. Then for each $\omega\in\Omega$,
\begin{equation}
\Lambda_{(\omega\mu-\iota,\lambda)}(z)=e^{\left\langle \mu,H\right\rangle
}\text{ }e^{2\pi i\left\langle \lambda,\lambda^{\vee}\right\rangle }.
\end{equation}
We rewrite (7.4.3) as $\Lambda_{(\mu,\lambda)}(z),\ $keeping in mind that in
general the character $\Lambda_{(\mu,\lambda)}$ is well-defined only on
$Z_{G}(\mathbb{R}).$ Then $\Lambda_{(\mu,\lambda)}$ is the central character
of each of $\pi\in\Pi;$ this is evident from the character formula for $\pi$
on $T(\mathbb{R})_{reg}.$

For consistency with the abelian case of the local Langlands correspondence,
\textit{i.e.} the case $G=T,$ we require that the renormalized packet $\Pi
_{D}$ consists of the discrete series representations attached to the
characters $(\Lambda_{(\omega\mu-\iota,\lambda)})_{D},$ $\omega\in\Omega.$
Recall from (4.2) of \cite{KS12}, if $\chi$ is a character on $T(\mathbb{R})$
then $\chi_{D}$ is defined to be $\chi^{-1}.$ In other words, we replace the
characters $\Lambda_{(\omega\mu-\iota,\lambda)}$ by $\Lambda_{(-\omega
\mu+\iota,-\lambda)}$ and the central character $\Lambda_{(\mu,\lambda)}$ by
$\Lambda_{(-\mu,-\lambda)}.$ Notice that this applies also to the parameters
that we have been ignoring, namely those that are essentially bounded but not bounded.

First we check that this prescription of $\Pi_{D}$ agrees with that in (4.4.)
of \cite{KS12}. Recall the automorphism of $^{L}G=G^{\vee}\rtimes
W_{\mathbb{R}}$ given by $^{L}\theta_{0}=\theta_{0}^{\vee}\times id$, where
$\theta_{0}^{\vee}$ (written in \cite{KS12} as $\theta_{0})$ denotes the
automorphism of $G^{\vee}$ that preserves $spl^{\vee}$ and acts on
$\mathcal{T}$ as $t\rightarrow\omega_{0}(t)^{-1}.$ Here $\omega_{0}$ is the
longest element of the Weyl group of $\mathcal{T}$. In particular, if $Z_{G}$
is anisotropic then $\theta_{0}^{\vee}$ is given by the action of the Galois
element $\sigma$ on $G^{\vee}$. According to \cite{KS12}, the packet $\Pi_{D}
$ attached to $\varphi$ is the classical packet for $^{L}\theta_{0}%
\circ\varphi.$ Since%
\begin{equation}
\theta_{0}^{\vee}(\mu)=-\omega_{0}\mu\text{, \ }\theta_{0}^{\vee}%
(\lambda)\equiv-\lambda\operatorname{mod}\mathcal{K},
\end{equation}
we see that $\Pi_{D}$ consists of the representations attached to
$\varphi(-\omega_{0}\mu,-\lambda)$ in the classical correspondence. The
character data for these coincide with that for the representations we have
prescribed above, and so we are done.

There is also another more convenient way to describe the parameter
$^{L}\theta_{0}\circ\varphi$. Namely we replace $spl^{\vee}$ by its opposite
$(spl^{\vee})^{opp}$, \textit{i.e.} we replace $\mathcal{B}$ by its opposite
relative to $\mathcal{T}$ and replace the root vector $X_{\alpha^{\vee}}$ by
the root vector $X_{-\alpha^{\vee}}$ that forms a simple triple with
$X_{\alpha^{\vee}}$ and the coroot $H_{\alpha^{\vee}}$ for $\alpha^{\vee}$.
Relative to this new splitting, $\varphi(-\mu,-\lambda)$ is the almost
canonical representative for $^{L}\theta_{0}\circ\varphi.$

\textit{Now fix a parameter} $\varphi$ and consider both the classical packet
$\Pi$ and the renormalized packet $\Pi_{D}\ $attached as we have described.
There are various ways to relate $\Pi_{D}$ directly to $\Pi;$ see also
\cite{AV12}, Section 4 of \cite{Ka13}. Since our representations are unitary
we will use complex conjugation. For each such $\pi$ we consider its complex
conjugate $\overline{\pi}$ on the conjugate Hilbert space; $\overline{\pi}$ is
the contragredient or dual of $\pi.$

Relative to given toral data, we have written a member $\pi$ of $\Pi$ as
$\pi(\omega),$ where $\omega$ represents a \textit{left} coset of
$\Omega_{\mathbb{R}}$ in $\Omega.$ By definition \cite{Sh10}, the character
formula for $\pi$ on $T(\mathbb{R})_{reg}$ involves only the characters
$\Lambda_{(\omega^{-1}\mu-\iota,\lambda)},$ for $\omega$ in this coset.
Relative to the same toral data and $(spl^{\vee})^{opp}$, we write a member
$\pi_{D}$ of $\Pi_{D}$ as $\pi_{D}(\omega)$ if its character formula on
$T(\mathbb{R})_{reg}$ involves only the characters $(\Lambda_{(\omega^{-1}%
\mu-\iota,\lambda)})_{D}.$

\begin{lemma}
(i) For each $\omega\in\Omega,$ if $\pi\in\Pi$ and $\pi=\pi(\omega)$ then
$\overline{\pi}\in\Pi_{D\text{ }}$ and $\overline{\pi}=\pi_{D}(\omega).$ (ii)
The map $\pi\rightarrow\pi_{D}=\overline{\pi}$ defines a bijection
$\Pi\rightarrow\Pi_{D}.$
\end{lemma}

\begin{proof}
From Harish-Chandra's character formula for $\overline{\pi}$ on $T(\mathbb{R}%
)_{reg}$ and his characterization of discrete series characters (see
\cite[Section 27]{HC75}), we see that $\overline{\pi}=\pi_{D}(\omega),$ and
the result follows.
\end{proof}

\begin{remark}
We could also argue Lemma 7.4.1 by applying Theorem 1.3 of \cite{AV12}.
\end{remark}

Finally, recall the stable characters $St$-$Trace$ $\pi,$ $St$-$Trace$
$\pi_{D}$ attached to $\Pi,\Pi_{D}$. Here each is simply the sum of the
characters of the members of the packet. Then the last lemma implies that%
\begin{equation}
St\text{-}Trace\text{ }\pi_{D}=\overline{St\text{-}Trace\text{ }\pi.}%
\end{equation}
Recall that for a distribution $\Theta$ and test function $f$ we have
$\overline{\Theta}(f)=\overline{\Theta(\overline{f})}$). Thus if
$St$-$Char(\pi,-),$ $St$-$Char(\pi_{D},-)$ denote the analytic functions on
the regular semisimple set of $G(\mathbb{R})$ that represent these two
distributions then
\begin{equation}
St\text{-}Char(\pi_{D},\delta)=\overline{St\text{-}Char(\pi,\delta)},
\end{equation}
for all regular semisimple $\delta\in G(\mathbb{R}).$

\subsection{Very regular related pairs and attached packets}

Consider a very regular related pair of bounded parameters $(\varphi
_{1},\varphi)$ as in (7.2). We first attach packets $\Pi_{1},\Pi$ in the
classical correspondence. For any $\pi_{1}\in\Pi_{1}$ and $\pi\in\Pi,$ we call
$(\pi_{1},\pi)$ a very regular related pair (of tempered irreducible
representations). Recall that the $G^{\vee}$-regularity of $\varphi_{1}$
ensures that if $\Pi_{1}$ consists of discrete series representations of
$H_{1}(\mathbb{R})$ then $\Pi$ consists of discrete series representations of
$G(\mathbb{R});$ see also \cite{Sh10}. More generally, we invoke parabolic
descent in standard endoscopy. Namely, to $(\varphi_{1},\varphi)$ we may
attach an endoscopic pair of Levi subgroups $M_{1},$ $M$ and discrete series
packets for both $M_{1}(\mathbb{R}),$ $M(\mathbb{R})\ $respectively. The
members of the corresponding packets for $H_{1}(\mathbb{R}),$ $G(\mathbb{R})$
are representations parabolically induced from representations in the packets
for $M_{1}(\mathbb{R}),$ $M(\mathbb{R})$ respectively.

Observe that Lemma 7.4.1 may be applied to all bounded parameters because
complex conjugation commutes in the obvious sense with unitary parabolic
induction in our setting (or again we could argue via \cite{AV12}). The
assertions following Lemma 7.4.1 on stable characters apply also.

Continuing with a very regular related pair of bounded parameters
$(\varphi_{1},\varphi),$ we also use the renormalized correspondence to attach
packets $\Pi_{1,D},\Pi_{D}$ to $(\varphi_{1},\varphi)$. Let $\pi_{1,D}\in
\Pi_{1,D}$ and $\pi_{D}\in\Pi_{D}.$ Then we call $(\pi_{1,D},\pi_{D})$ a very
regular related pair also. Thus if $(\pi_{1},\pi)$ is a very regular related
pair in the first (classical) sense then $(\pi_{1,D},\pi_{D})=(\overline
{\pi_{1}},\overline{\pi})$ is a very regular related pair in the second
(renormalized) sense, with the converse also true. We will use this throughout
Sections 8 and 9 without further comment.

\section{\textbf{Archimedean case: transfer theorems}}

\subsection{Spectral relative factors $\Delta^{\prime}$ and $\Delta_{D}$}

We work in the setting of very regular related pairs of bounded parameters; we
take two such pairs and consider the various attached packets. The relative
terms $\Delta_{I},\Delta_{II},\Delta_{III}$ were defined in \cite{Sh10} for
related pairs in the classical sense.

The term $\Delta_{I}$ is a sign given by Tate-Nakayama pairing. Thus we use
the same definition of $\Delta_{I}$ for the renormalized related pairs (see
Section 8 of \cite{Sh10}), so that
\begin{equation}
\Delta_{I}(\pi_{1,D},\pi_{D})=\Delta_{I}(\pi_{1},\pi)
\end{equation}
for any $(\pi_{1},\pi),(\pi_{1,D},\pi_{D})$ attached to given very regular
related pair of parameters.

In the case of $\Delta_{III}$, another Tate-Nakayama sign, we follow Section
10 of \cite{Sh10}. Then we conclude that%
\begin{equation}
\Delta_{III}(\pi_{1},\pi;\pi_{1}^{\dag},\pi^{\dag})=\Delta_{III}(\overline
{\pi_{1}},\overline{\pi};\overline{\pi_{1}^{\dag}},\overline{\pi^{\dag}})
\end{equation}
for all very regular related pairs $(\pi_{1},\pi),$ $(\pi_{1}^{\dag},\pi
^{\dag}).$

In Section 9 of \cite{Sh10}, $\Delta_{II}$ is defined as a quotient of
absolute terms that are fourth roots of unity. We now define the term
$\Delta_{II,D}$ in such a way that%
\begin{equation}
\Delta_{II,D}(\overline{\pi_{1}},\overline{\pi})=\Delta_{II}(\pi_{1},\pi
)^{-1}=\overline{\Delta_{II}(\pi_{1},\pi)},
\end{equation}
for all very regular related pairs $(\pi_{1},\pi).$ Namely we proceed as in
Section 9 of \cite{Sh10}, but use $-\mu_{1},-\iota_{1},-\lambda_{1}$ in place
of $\mu_{1},\iota_{1},\lambda_{1}$ in character formulas$.\ $We also replace
$\mu^{\ast},\lambda^{\ast}$ described there by $-\mu^{\ast},-\lambda^{\ast}$.
Notice that this is done also when we define explicitly the geometric
$\Delta_{III_{2},D}$ of (3.2.1). Then the pair denoted $\mu,\lambda$ in
\cite{Sh10}, namely $\mu=\mu_{1}+\mu^{\ast},\lambda=\lambda_{1}+\lambda^{\ast
},$ is replaced by $-\mu,-\lambda$. There $\mu$ is not necessarily dominant.
In the case that $\mu$ is not dominant there is an additional sign
contribution to $\Delta_{II}$; we use the same one in $\Delta_{II,D}.$ If
$\mu$ is dominant then we call $\varphi_{1}$ \textit{well-positioned} relative
to $\varphi$; the discussion of Section 7c of \cite{Sh10} shows that this is a
well-defined notion. Notice also that $\Delta_{II},\Delta_{II,D}$ depend only
on the pair of parameters and not on the choice among attached representations
since only the stable characters are used in the definition. Recall that the
stable character of $\Pi_{1,D}$, $\Pi_{D}$ respectively is the complex
conjugate of that of $\Pi_{1},\Pi$ respectively. So far, we have arranged to
replace everything in each character formula in Section 9 of \cite{Sh10} by
its complex conjugate except for the constant, a fourth root of unity, that
contributes to $\Delta_{II}(\pi_{1},\pi)$. We thus also replace $\Delta
_{II}(\pi_{1},\pi)$ by its complex conjugate, as desired. Recall that the
\textit{relative} version of $\Delta_{II}$ is a sign \cite{Sh10}.

As usual for the real case, we identify $\Delta^{\prime}$ with $\Delta
=\Delta_{I}\Delta_{II}\Delta_{III}$. Next, set
\begin{equation}
\Delta_{D}=\Delta_{I}\Delta_{II,D}\Delta_{III}.
\end{equation}
The following is now immediate.

\begin{lemma}
Suppose that $(\pi_{1},\pi)$ and $(\pi_{1}^{\dag},\pi^{\dag})$ are very
regular related pairs. Then we have an equality of signs:%
\begin{equation}
\Delta(\pi_{1},\pi;\pi_{1}^{\dag},\pi^{\dag})=\Delta_{D}(\overline{\pi_{1}%
},\overline{\pi};\overline{\pi_{1}^{\dag}},\overline{\pi^{\dag}}).
\end{equation}

\end{lemma}

We emphasize that (8.1.5) takes the form given because all the critical
contributions are signs.

\subsection{Spectral absolute $\Delta^{\prime}$ and $\Delta_{D}$}

We define absolute spectral factors $\Delta^{\prime}(\pi_{1},\pi)$,
$\Delta_{D}(\pi_{1,D},\pi_{D})$ following the same method as for the geometric
factors $\Delta^{\prime}(\gamma_{1},\delta)$, $\Delta_{D}(\gamma_{1},\delta)$.
In particular, if the very regular pair $(\varphi_{1},\varphi)$ is not related
then $\Delta^{\prime}$ and $\Delta_{D}$ vanish on the attached packets; more
precisely,
\begin{equation}
\Delta^{\prime}(\pi_{1},\pi)=\Delta_{D}(\pi_{1,D},\pi_{D})=0
\end{equation}
for all $\pi_{1}\in\Pi_{1},$ $\pi\in\Pi,$ $\pi_{1,D}\in\Pi_{1,D}$ and $\pi
_{D}\in\Pi_{D}$.

Now consider just the factors $\Delta^{\prime}$ for the classical local
Langlands correspondence. It is clear that for a given choice of geometric
factor $\Delta^{\prime}$ there is at most one choice of spectral factor
$\Delta^{\prime}$ for which the dual spectral transfer statement (7.1.1) is
true, and conversely the choice of spectral factor fixes the choice of
geometric factor. To describe which normalizations are compatible in this
sense we introduce yet another canonical factor, the compatibility factor
$\Delta^{\prime}(\pi_{1},\pi;\gamma_{1},\delta)$ (see Section 12 of
\cite{Sh10}). Here we again identify $\Delta^{\prime}$ with $\Delta=\Delta
_{I}\Delta_{II}\Delta_{III},$ and observe that only the term $\Delta_{III}%
(\pi_{1},\pi;\gamma_{1},\delta)$ is genuinely relative. Then we call
$\Delta^{\prime}(\gamma_{1},\delta)$ and $\Delta^{\prime}(\pi_{1},\pi)$
compatible if and only if
\begin{equation}
\Delta^{\prime}(\pi_{1},\pi)=\Delta^{\prime}(\pi_{1},\pi;\gamma_{1}%
,\delta)\text{ }\Delta^{\prime}(\gamma_{1},\delta)
\end{equation}
for some choice of very regular related pairs $(\pi_{1},\pi),(\gamma
_{1},\delta).$ Then it is true for all such pairs by the transitivity
properties (ii) and (iii) of Lemma 12.1 in \cite{Sh10}.

To consider compatibility of the factors $\Delta_{D}(\gamma_{1},\delta)$ and
$\Delta_{D}(\pi_{1,D},\pi_{D})$ we introduce $\Delta_{D}(\pi_{1,D},\pi
_{D};\gamma_{1},\delta),$ using once again the paradigm $\Delta_{D}$
$=\Delta_{I}\Delta_{II,D}\Delta_{III},$ and then we proceed in the same way as
for $\Delta^{\prime}.$ Following the arguments for Corollary 5.0.3 and Lemma
8.1.1 we observe that%
\begin{equation}
\Delta_{D}(\overline{\pi_{1}},\overline{\pi};\gamma_{1},\delta)=\overline
{\Delta^{\prime}(\pi_{1},\pi;\gamma_{1},\delta)},
\end{equation}
for all very regular related pairs $(\pi_{1},\pi),(\gamma_{1},\delta).$

\begin{lemma}
Let $\Delta^{\prime}(\gamma_{1},\delta)$, $\Delta_{D}(\gamma_{1},\delta)$ be
absolute transfer factors and let $c\in\mathbb{C}^{\times}$ be such that
\begin{equation}
\Delta_{D}(\gamma_{1},\delta)=c\text{ }\overline{\Delta^{\prime}(\gamma
_{1},\delta)}%
\end{equation}
for all very regular pairs $(\gamma_{1},\delta).$ Suppose that $\Delta
^{\prime}(\pi_{1},\pi)$ is compatible with $\Delta^{\prime}(\gamma_{1}%
,\delta)$ and that $\Delta_{D}(\pi_{1,D},\pi_{D})$ is compatible with
$\Delta_{D}(\gamma_{1},\delta).$ Then%
\begin{equation}
\Delta_{D}(\overline{\pi_{1}},\overline{\pi})=c\overline{\Delta^{\prime}%
(\pi_{1},\pi)}%
\end{equation}
for all very regular pairs $(\pi_{1},\pi).$ Conversely, (8.2.4) and (8.2.5)
imply that either both or neither of $\Delta^{\prime},\Delta_{D}$ have
geometric-spectral compatibility.
\end{lemma}

Recall that Corollary 6.1.2 proves the existence of $c$ for given factors
$\Delta^{\prime}(\gamma_{1},\delta)$, $\Delta_{D}(\gamma_{1},\delta).$

\begin{proof}
This is a consequence of (8.2.3).
\end{proof}

\subsection{Quasi-split case: spectral $\Delta_{0},\Delta_{0,D}$ and
$\Delta_{\lambda},\Delta_{\lambda,D}$}

If $G$ is quasi-split over $\mathbb{R}$ we may as well take $G=G^{\ast}$ and
the inner twist $\psi$ to be the identity automorphism of $G.$ We assemble
spectral factors $\Delta_{0},\Delta_{0,D}$ following the paradigm for the
geometric case in (6.2). These spectral factors are compatible with the
corresponding geometric ones; see Lemma 12.3 in \cite{Sh10} for $\Delta_{0}$
and argue similarly for $\Delta_{0,D}$ or use (8.2.3) above. We do the same
for the Whittaker factors $\Delta_{\lambda},\Delta_{\lambda,D}.$

\begin{lemma}
For all very regular pairs $(\pi_{1},\pi)$ we have%
\begin{equation}
\Delta_{\lambda,D}(\overline{\pi_{1}},\overline{\pi})=c\text{ }\overline
{\Delta_{\lambda}(\pi_{1},\pi)},
\end{equation}
where $c=(det(V))(-1)$ as in Lemma 6.3.1, and
\begin{equation}
\Delta_{\overline{\lambda},D}(\overline{\pi_{1}},\overline{\pi})=\overline
{\Delta_{\lambda}(\pi_{1},\pi)}.
\end{equation}

\begin{proof}
This follows from Lemmas 6.3.1 and 6.3.2, geometric-spectral compatibility for
each of the factors $\Delta_{\lambda},\Delta_{\lambda,D}$ and (8.2.3).
\end{proof}
\end{lemma}

See (8.4.3) below to extend this to all bounded related pairs.

\subsection{Consequences for spectral transfer}

Suppose we have compatible geometric and spectral factors $\Delta^{\prime}$
and that the test functions $f\in C_{c}^{\infty}(G(\mathbb{R}))$ and $f_{1}\in
C_{c}^{\infty}(H_{1}(\mathbb{R}),\lambda_{H_{1}})$ have $\Delta^{\prime}%
$-matching orbital integrals as in (6.4.1). Here we may replace $C_{c}%
^{\infty}(-)$ by the corresponding Harish-Chandra-Schwartz space
$\mathcal{C}(-)$ if we wish. Then $f$ and $f_{1}$ also have $\Delta^{\prime}
$-matching traces on the very regular pairs of irreducible tempered
representations $(\pi_{1},\pi)$ as in (7.1.1). A proof of this, which comes
from \cite{Sh82} and \cite{LS90}, is reviewed in detail in \cite{Sh10}.

Suppose instead we work with compatible geometric and spectral factors
$\Delta_{D}.$ Returning to (6.4), suppose as above that $f$ is a test function
on $G(\mathbb{R}).$ Use geometric $\Delta^{\prime}$-transfer for $\overline
{f}$ to define $f_{1,D}$ as in (6.4.2) such that $f$ and $f_{1,D}$ have
$\Delta_{D}$-matching orbital integrals. We claim that $f$ and $f_{1,D} $ also
have $\Delta_{D}$-matching traces on the very regular pairs of irreducible
tempered representations $(\pi_{1,D},\pi)$ as in (7.1.2). For this, we will
change notation on the right side of (7.1.2), writing $\pi_{D} $ in place of
$\pi$ and then setting $\pi=\overline{\pi_{D}},$ $\pi_{1}=\overline{\pi_{1,D}%
}.$ We observe that%
\begin{equation}
\Delta_{D}(\pi_{1,D},\pi_{D})\text{ }Trace\text{ }\pi_{D}(f)=c\text{
}\overline{\Delta(\pi_{1},\pi)\text{ }Trace\text{ }\pi(\overline{f})}%
\end{equation}
and
\begin{equation}
St\text{-}Trace\text{ }\pi_{1,D}(f_{1,D})=c\text{ }\overline{St\text{-}%
Trace\text{ }\pi_{1}((\overline{f})_{1})}.
\end{equation}
Thus the claim follows. Conversely if we start with $\Delta_{D}$-transfer then
we may deduce $\Delta^{\prime}$-transfer.

We now drop the \textit{very regular} assumption on the spectral side. The
transfer proved on the geometric side provides a correspondence on test
functions that is full in the sense that it is sufficient to establish the
dual transfer to a tempered invariant eigendistribution on $G(\mathbb{R})$ of
\textit{any} stable tempered trace $St$-$Trace$ $\pi_{1,D}$ on $H_{1}%
(\mathbb{R})$ for which $Z_{1}(\mathbb{R})$ acts correctly. It remains then to
define the relevant spectral transfer factors and identify the distribution as
the right side of (7.1.2). We could proceed directly, copying the approach for
$\Delta^{\prime}$-transfer described in \cite{Sh10}. That involves long and
delicate arguments with limits of discrete series representations, although
the adjustments needed are minor. Alternatively, we may invoke the results for
$\Delta^{\prime}$-transfer and use the formula (8.2.5) to extend the
definition of $\Delta_{D}(\pi_{1,D},\pi_{D})$ by
\begin{equation}
\Delta_{D}(\pi_{1,D},\pi_{D})=c\text{ }\overline{\Delta^{\prime}(\overline
{\pi_{1,D}},\overline{\pi_{D}})},
\end{equation}
where the constant $c$ is given by (8.2.4). Notice that Lemma 8.3.1 remains
true. Now we deduce (7.1.2) from (7.1.1) as in the very regular case. In conclusion:

\begin{lemma}
Suppose test functions $f$ on $G(\mathbb{R})$ and $f_{1,D}$ on $H_{1}%
(\mathbb{R})$ have $\Delta_{D}$-matching orbital integrals. Then
\begin{equation}
St\text{-}Trace\text{ }\pi_{1,D}(f_{1,D})=\sum_{\pi_{D}}\text{ }\Delta_{D}%
(\pi_{1,D},\pi_{D})\text{ }Trace\text{ }\pi(f),
\end{equation}
for each tempered irreducible admissible representation $\pi_{1,D}$ of
$H_{1}(\mathbb{R})$ such that the central subgroup $Z_{1}(\mathbb{R})$ acts by
the character $\lambda_{H_{1},D},$ where the summation is over tempered
irreducible admissible representations $\pi_{D}$ of $G(\mathbb{R}).$
\end{lemma}

\subsection{Extended groups and extended packets}

Turning to structure on stable classes and packets, we recall that if the
simply-connected covering $G_{sc}$ of the derived group of $G$ has nontrivial
first Galois cohomology set $H^{1}(\Gamma,G_{sc})$ then stable conjugacy
classes of regular elliptic elements have insufficient conjugacy classes, and
discrete series packets have insufficient members, for an exact duality with
an evident candidate. In the notation of (7.3), this candidate is the group of
$\overline{\sigma}$-invariants in the quotient of $\mathcal{T}$ by the center
of $G^{\vee},$ and the evidence is Tate-Nakayama duality.

As Vogan first pointed out, such problems may be resolved by using several
inner forms simultaneously in place of the single group $G$. We follow
Kottwitz's formulation of an extended group, also called a $K$-group
\cite{Ar99}. See Section 4 of \cite{Sh08} for a review with examples; unitary
groups are discussed at length in the reference [S7] of that paper, available
at the author's website. This solves our problem minimally and is useful in
stabilization of the Arthur-Selberg trace formula \cite{Ar06, Ar13}.

Both geometric and spectral factors extend naturally to this setting. See
Section 5 of \cite{Sh08} for a discussion of $\Delta=\Delta^{\prime}$. The
discussion for $\Delta_{D}$ parallels that for $\Delta^{\prime}$ in the same
manner as for a single group. We will assume from now on, \textit{without
changing notation}, that we are in the extended group setting.

\section{\textbf{Archimedean case: applications}}

\subsection{Structure on tempered packets}

Here we recall that our factors $\Delta=\Delta^{\prime}$ give structure on
tempered packets in the form proposed by Arthur in \cite{Ar06}, remark that
there is an analog for $\Delta_{D},$ and comment on the relation between the two.

Before starting, we should note that for extended groups without a quasi-split
component, recent work of Kaletha \cite{Ka13a} allows us to replace
$\mathcal{S}_{\varphi}^{sc}$ and $\zeta$ with better behaved objects, but only
after modifying the notion of extended group. Our present elementary approach
will be sufficient for the applications in \cite{Sh} and will simplify some
already delicate computations.

Let $\boldsymbol{\varphi}$ be a bounded Langlands parameter for the extended
group $G.$ As usual, $\varphi$ will denote a representative for
$\boldsymbol{\varphi}\mathbf{.}$ We consider the group $S_{\varphi}^{ad}$
consisting of the images in $G_{ad}^{\vee}=(G^{\vee})_{ad}$ of those elements
in $G^{\vee}$ which fix the image of $\varphi$ under the conjugation action.
Write $\mathcal{S}_{\varphi}$ for the component group of $S_{\varphi}^{ad}.$
Then $\mathcal{S}_{\varphi}$ is a finite abelian group, in fact a sum of
groups of order two.

A change to representative $\varphi^{\prime}$ determines a unique isomorphism
$\mathcal{S}_{\varphi}\rightarrow\mathcal{S}_{\varphi^{\prime}}.$ In the
quasi-split case the group $\mathcal{S}_{\varphi}$ is sufficient for our
purposes, but in general we introduce also $S_{\varphi}^{sc},$ the inverse
image of $S_{\varphi}^{ad}$ in $G_{sc}^{\vee}=(G^{\vee})_{sc}.$ Then
$\mathcal{S}_{\varphi}^{sc}$ will be the component group of $S_{\varphi}%
^{sc}.$ In Section 7 of \cite{Sh08} we attach to $s\in\mathcal{S}_{\varphi
}^{sc}$ an endoscopic group and parameter $\varphi^{(s)}$ for that group such
that $(\varphi^{(s)},\varphi)$ is a related pair. Moreover, $\varphi^{(s)}$ is
well-positioned relative to $\varphi$ (see (8.1) above). Let $\Pi^{(s)},\Pi$
denote the corresponding classical packets. Let $\pi^{(s)}\in\Pi^{(s)}$; the
choice does not matter. We also pick a basepoint $\pi^{base}$ for $\Pi$ and
character $\zeta$ on $\mathcal{S}_{\varphi}^{sc}$ as specified below. Then
\begin{equation}
\left\langle s,\pi^{base}\right\rangle =\zeta(s)
\end{equation}
and
\begin{equation}
\left\langle s,\pi\right\rangle \diagup\left\langle s,\pi^{\dag}\right\rangle
=\Delta(\pi^{(s)},\pi;\pi^{(s)},\pi^{\dag}),
\end{equation}
for all $s\in\mathcal{S}_{\varphi}^{sc}$ and $\pi,\pi^{\dag}\in\Pi,$ together
define a pairing
\begin{equation}
\left\langle -,-\right\rangle :\mathcal{S}_{\varphi}^{sc}\times\Pi
\rightarrow\mathbb{C}^{\times}%
\end{equation}
that identifies $\Pi$ as the set of characters $\zeta\otimes\chi,$ where
$\chi$ is a character on $\mathcal{S}_{\varphi}^{sc}$ trivial on the kernel of
$\mathcal{S}_{\varphi}^{sc}\rightarrow\mathcal{S}_{\varphi}.$ For all this, as
well as the requirements on $\zeta,$ see Sections 7 and 11 of \cite{Sh08}. The
main result there is Theorem 7.5 and its reinterpretation in Corollary 11.1.

Thus $\Pi$ is identified as the dual of $\mathcal{S}_{\varphi}$ only up to a
twist. If we turn now to $\Delta_{D}$ and the renormalized packets $\Pi
_{D}^{(s)},\Pi_{D}$ attached to the related pair $(\varphi^{(s)},\varphi), $
we may take $\pi_{D}^{base}=\overline{\pi^{base}}$ and reuse $\zeta$ if we
wish because the character $\zeta_{G}$ of Corollary 11.1 of \cite{Sh08} has
order one or two and so coincides with its renormalized version. Then, by
Lemma 8.1.1 above, we get the pairing
\begin{equation}
\left\langle -,-\right\rangle _{D}:\mathcal{S}_{\varphi}^{sc}\times\Pi
_{D}\rightarrow\mathbb{C}^{\times}%
\end{equation}
given by
\begin{equation}
\left\langle s,\pi_{D}\right\rangle _{D}=\left\langle s,\pi\right\rangle ,
\end{equation}
for all $s\in\mathcal{S}_{\varphi}^{sc}$ and $\pi_{D}\in\Pi_{D}$, where
$\pi=\overline{\pi_{D}}.$ We could also have used $\overline{\zeta}$ or a
different basepoint, etc.. In short, Lemma 8.1.1 allows us to apply the
results in Sections 7 and 11 of \cite{Sh08} to renormalized packets.

In the quasi-split case, \textit{i.e.} the case that the extended group has a
quasi-split component which is then unique \cite{Sh08}, we may return to
$\mathcal{S}_{\varphi}$ as described separately in the next subsection.

\subsection{Whittaker normalizations and strong base-point property}

We first observe the strong base-point property for the renormalized Whittaker
factors $\Delta_{\lambda,D}.$

\begin{lemma}
Let $\varphi$ be a bounded Langlands parameter for the extended group $G$ of
quasi-split type and denote by $\Pi_{D}$ the renormalized extended packet
attached to $\varphi.$ Let $\lambda$ be a set of Whittaker data for $G\ $and
suppose $\pi_{D}$ is the unique member of $\Pi_{D}$ that is $\lambda$-generic.
Then%
\begin{equation}
\Delta_{\lambda,D}(\pi_{D}^{(s)},\pi_{D})=1
\end{equation}
for all $s\in\mathcal{S}_{\varphi}.$
\end{lemma}

\begin{proof}
The parallel statement for the factors $\Delta_{\lambda}^{\prime}%
=\Delta_{\lambda}$ and the classical extended packet $\Pi$ attached to
$\varphi$ is proved as Theorem 11.5 in \cite{Sh08}. Then (8.3.2) finishes the proof.
\end{proof}

Continuing with $\Pi,$ the classical extended packet attached to $\varphi,$ we
recall from Lemma 11.4 and Corollary 14.1 of \cite{Sh10} that
\begin{equation}
\Delta_{\lambda}(\pi^{(s)},\pi)=\pm1,
\end{equation}
for all $s\in\mathcal{S}_{\varphi},\pi\in\Pi.$ Then, by (8.3.2) again,%
\begin{equation}
\Delta_{\lambda,D}(\pi_{D}^{(s)},\pi_{D})=\pm1,
\end{equation}
for all $s\in\mathcal{S}_{\varphi},\pi_{D}\in\Pi_{D}.$

We have further from Theorem 7.5 and Corollary 11.1 of \cite{Sh08} that the
pairing
\begin{equation}
\left\langle -,-\right\rangle _{\lambda}:\mathcal{S}_{\varphi}\times
\Pi\rightarrow\{\pm1\},
\end{equation}
given by$^{{}}$%
\begin{equation}
(s,\pi)\rightarrow\Delta_{\lambda}(\pi^{(s)},\pi)
\end{equation}
is perfect in the sense that it identifies $\Pi$ as the dual of the finite
abelian group $\mathcal{S}_{\varphi}.$ From (8.3.2) yet again, we conclude:

\begin{lemma}
The map%
\begin{equation}
\left\langle -,-\right\rangle _{\lambda,D}:\mathcal{S}_{\varphi}\times\Pi
_{D}\rightarrow\{\pm1\}
\end{equation}
given by
\begin{equation}
(s,\pi_{D})\rightarrow\Delta_{\lambda,D}(\pi_{D}^{(s)},\pi_{D})
\end{equation}
is a perfect pairing of $\Pi_{D}$ with $\mathcal{S}_{\varphi},$ and moreover,%

\begin{equation}
\left\langle s,\pi_{D}\right\rangle _{\lambda,D}=\left\langle s,\pi
\right\rangle _{\overline{\lambda}},
\end{equation}
for all $s\in\mathcal{S}_{\varphi},\pi_{D}\in\Pi_{D},$ where $\pi
=\overline{\pi_{D}}\in\Pi.$
\end{lemma}

The sign $\left\langle s,\pi_{D}\right\rangle _{\lambda,D}$ is easily computed
explicitly via the Tate-Nakayama pairing when $\varphi$ is regular or, more
generally, when $\Pi$ is presented with nondegenerate data. On the other hand,
see \cite{Sh08, Sh14} for the \textit{totally degenerate} limits of discrete
series that exist, for example, when the derived group of $G$ is
simply-connected. This case, which includes the representations studied by
Carayol and Knapp \cite{CK07}, is also easy to calculate since each member of
the packet is generic and the sets of Whittaker data index the packet. We take
this up in \cite{Sh14}.

\end{document}